\documentclass[11pt]{amsart}

\usepackage{amsfonts,amssymb,amsmath,amsthm}
\usepackage{bbm}
\usepackage{hhline}
\usepackage{stmaryrd}
\usepackage{longtable}

\input{xy}
\xyoption{all}

\long\def\symbolfootnote[#1]#2{\begingroup%
\def\thefootnote{\fnsymbol{footnote}}\footnote[#1]{#2}\endgroup}

\begin{document}

\setlength{\parindent}{0pt} \setlength{\parskip}{1ex}

\newtheorem{Le}{Lemma}[section]
\newtheorem{Pro}[Le]{Proposition}
\newtheorem{Th}[Le]{Theorem}
\newtheorem{Co}[Le]{Corollary}
\newtheorem{Conj}[Le]{Conjecture}

\theoremstyle{definition}
\newtheorem{Def}[Le]{Definition}
\newtheorem{DaL}[Le]{Definition and Lemma}
\newtheorem{Conv}[Le]{Convention}
\newtheorem{Assumpt}[Le]{Assumption}

\theoremstyle{remark}
\newtheorem{Ex}[Le]{Example}
\newtheorem{Rem}[Le]{Remark}

\author{Frank Reidegeld}
\title[Exceptional holonomy on vector bundles]{Exceptional holonomy 
on vector bundles with two-dimensional fibers}
\address{E-Mail: frank.reidegeld\char"40 math.tu-dortmund.de}

\begin{abstract}
An $SU(3)$- or $SU(1,2)$-structure on a 6-dimensional manifold
$N^6$ can be defined as a pair of a 2-form $\omega$ and a 3-form
$\rho$. We prove that any analytic $SU(3)$- or $SU(1,2)$-structure 
on $N^6$ with $d\omega \wedge \omega =0$ can be extended to a
parallel Spin($7$)- or $\text{Spin}_0(3,4)$-structure $\Phi$ that is 
defined on the trivial disc bundle $N^6\times B_{\epsilon}(0)$ for a 
sufficiently small $\epsilon>0$. Furthermore, we show by an example 
that $\Phi$ is not uniquely determined by $(\omega,\rho)$ and discuss
if our result can be generalized to non-trivial bundles.
\end{abstract}

\maketitle

\symbolfootnote[0]{The author was supported by the 
Individual Grant RE 3147/1-1 of the \emph{Deutsche Forschungsgemeinschaft}
while working on this article.}

\setcounter{tocdepth}{1}
\tableofcontents

\section{Introduction}

In his article on stable forms, Hitchin \cite{Hitchin} proposed
a new method to construct manifolds with exceptional holonomy. The
starting point of his construction is a 7-dimensional manifold $M$
with a $G_2$-structure $\phi$ that satisfies $d\ast\phi=0$. We can 
take $\phi$ as an initial value for a certain flow equation such that the solution 
of the initial value problem yields a parallel Spin($7$)-structure on $M\times 
(-\epsilon,\epsilon)$ for an $\epsilon>0$. This idea can be generalized to 
the semi-Riemannian case where we obtain a parallel 
$\text{Spin}_0(3,4)$-structure \cite{Cortes}.

Many of the known complete metrics with holonomy Spin($7$) are not defined on
a manifold of type $M\times (-\epsilon,\epsilon)$ but on a disc bundle
over a lower-dimensional manifold \cite{Baz,Baz2,BrS,Cve,KanI,KanII,Rei}. 
The reason behind this is that those metrics are
of cohomogeneity one and that the cohomogeneity-one manifolds of this
type are the only ones that admit complete metrics with holonomy Spin($7$)
\cite{Rei}.

Bielawski \cite{Biel} proves another result that fits into this context. Let $X$ 
be a real analytic K\"ahler manifold. We identify $X$ with the zero section of 
its canonical bundle. The K\"ahler metric on $X$ can be uniquely 
extended to a Ricci-flat K\"ahler metric on a neighborhood of $X$ such that 
the $U(1)$-action on the bundle is isometric and Hamiltonian. We thus
have extended the $U(n)$-structure on the base to an $SU(n+1)$-structure
on the bundle.

Motivated by these facts, we attempt to construct parallel Spin($7$)- or\linebreak
$\text{Spin}_0(3,4)$-structures on $\mathbb{R}^2$-bundles. More 
precisely, let $(\omega,\rho)$ be a pair of a 2-form and a 3-form on 
a 6-dimensional manifold $N^6$ that defines an $SU(3)$- or 
$SU(1,2)$-structure. We search for conditions on $(\omega,\rho)$ 
such that on $M^8:=N^6 \times B_{\epsilon}(0)$, where $B_{\epsilon}(0)$ 
is a ball of radius $\epsilon>0$ in $\mathbb{R}^2$, there exists a 
pa\-rallel Spin($7$)- or $\text{Spin}_0(3,4)$-structure that extends
in a suitable sense the $G$-structure $(\omega,\rho)$. We also
discuss the case where $M^8$ is a bundle over $N^6$ with 
$B_{\epsilon}(0)$ as fiber. 

The problem of how to extend a geometric structure on an 
$(n-1)$-dimensional manifold to a manifold of dimension $n$ with
special holonomy or another kind of special geometry has been 
extensively studied in the literature \cite{Br2010},\cite{Conti},\cite{ContiS},
\cite{Cortes},\cite{Hitchin},\cite{Stock1},\cite{Stock2}. To our best
knowledge the case where the codimension is $2$ is dealt with only in
\cite{Biel} and the present paper. 

The article is organized as follows. In Section 2 and 3 we give
an introduction to the $G$-structures that we need and to Hitchin's
flow equation. We set up our initial value problem and prove that 
it has a local solution in the following section. After that we show
with help of an example that our solution can be non-unique.
In the sixth section, we finally discuss if our result can be generalized
to non-trivial bundles over 6-dimensional manifolds.

\section{$G$-structures}

\subsection{$G$ is $SU(3)$ or $SU(1,2)$}
\label{SU3Subsection}

In order to prove our theorem we have to introduce several
$G$-structures. We start with $G$-structures on 6-dimensional
manifolds and then proceed to the 7- and 8-dimensional case.
A well written introduction to all of these $G$-structures can be 
found in Cort\'{e}s et al. \cite{Cortes}. We use similar conventions 
as \cite{Cortes} and only recapitulate the facts that we need for 
our considerations. Although a $G$-structure is in general defined 
as a principal bundle, all $G$-structures in this section can be 
described with help of certain differential forms. Throughout this 
article we use the following convention.

\begin{Conv}
Let $(v_i)_{i\in I}$ be a basis of a vector space $V$. We denote
its dual basis by $(v^i)_{i\in I}$ and abbreviate $v^{i_1} \wedge
\ldots \wedge v^{i_k}$ by $v^{i_1\ldots i_k}$.
\end{Conv}

Let $(e_i)_{i=1,\ldots,6}$ be the canonical basis of
$\mathbb{R}^6$. We define the 2-forms

\begin{equation}
\label{omega1} \omega_{SU(3)} := e^{12} + e^{34} + e^{56}
\end{equation}

and

\begin{equation}
\label{omega2} \omega_{SU(1,2)} := -e^{12} - e^{34} + e^{56}\:.
\end{equation}

Moreover, we introduce the canonical 3-form

\begin{equation}
\label{rhoc} \rho_{can.} := e^{135} - e^{146} - e^{236} -
e^{245}\:.
\end{equation}

The following lemma is proven in \cite{Cortes}.

\begin{Le}
\label{SU3LemmaI} Let $G\in \{SU(3), SU(1,2)\}$. The subgroup of
all $A\in GL(6,\mathbb{R})$ that stabilize $\omega_G$ and
$\rho_{can.}$ simultaneously is isomorphic to $G$.
\end{Le}

This motivates the following definition.

\begin{Def}
\label{SU3Def} Let $G\in \{SU(3), SU(1,2)\}$, $V$ be a
6-dimensional real vector space and $(\omega,\rho)$ be a pair of a
2-form and a 3-form on $V$. If there exists a basis
$(v_i)_{i=1,\ldots,6}$ of $V$ such that with respect to this basis
$\omega$ can be identified with $\omega_G$ and $\rho$ with
$\rho_{can.}$, $(\omega,\rho)$ is called a \emph{$G$-structure}.
\end{Def}

Hitchin \cite{Hitchin} has introduced the notion of a stable form.

\begin{Def}
Let $V$ be a real or complex vector space and $\beta\in
\bigwedge^k V^{\ast}$ with $k\in\{0,\ldots,\dim{V}\}$ be a
$k$-form. $\beta$ is called \emph{stable} if the $GL(V)$-orbit of
$\beta$ is an open subset of $\bigwedge^k V^{\ast}$.
\end{Def}

\begin{Le}
Let $(\omega,\rho)$ be a $G$-structure where $G\in \{SU(3),
SU(1,2)\}$. In this situation, $\omega$ and $\rho$ are both stable
forms.
\end{Le}

\begin{Rem}
The stable forms are an open dense subset of $\bigwedge^2
\mathbb{R}^{6\ast}$ and of $\bigwedge^3 \mathbb{R}^{6\ast}$.
There is exactly one open $GL(6,\mathbb{R})$-orbit in $\bigwedge^2
\mathbb{R}^{6\ast}$ and two open orbits in $\bigwedge^3
\mathbb{R}^{6\ast}$. One of them is the orbit of $\rho_{can.}$.
The other one can be used to define the notion of an
$SL(3,\mathbb{R})$-structure, which we will not consider in this
article.
\end{Rem}

Let $V$ be a 6-dimensional real vector space and $\bigwedge^k_s
V^{\ast}$ be the set of all stable $k$-forms on $V$. We can assign
to any $\rho\in \bigwedge^3_s V^{\ast}$ a certain endomorphism
$J_{\rho}$ by a map

\begin{equation}
i: {\bigwedge}^3_s V^{\ast}  \rightarrow  V\otimes V^{\ast}\:.
\end{equation}

$i$ is a rational $GL(6,\mathbb{R})$-equivariant map and is
described in detail in \cite{Cortes}. $i(\rho_{can.})$ is the
canonical complex structure on $\mathbb{R}^6$ which maps
$e_{2i-1}$ to $-e_{2i}$ and $e_{2i}$ to $e_{2i-1}$ for all
$i\in\{1,2,3\}$. If $(\omega,\rho)$ is an $SU(3)$- or an
$SU(1,2)$-structure, $J_{\rho}$ is a complex structure, too.
With help of another map

\begin{equation}
j: {\bigwedge}^2_s V^{\ast} \times {\bigwedge}^3_s V^{\ast}
\rightarrow S^2(V^{\ast})
\end{equation}

we can assign to $(\omega,\rho)$ a symmetric non-degenerate
bilinear form. $j$ is also a rational $GL(6,\mathbb{R})$-equivariant
map that is described explicitly in \cite{Cortes}. If $(\omega,\rho)$ 
is an

\begin{enumerate}
    \item $SU(3)$-structure, $j(\omega,\rho)$ is a metric with
    signature $(6,0)$. In particular, $j(\omega_{SU(3)},\rho_{can.})$
    is the Euclidean metric on $\mathbb{R}^6$.
    \item $SU(1,2)$-structure, $j(\omega,\rho)$ is a metric with
    signature $(2,4)$. In particular,

    \vspace{-11pt}

    \begin{equation}
    \label{24metric}
    \begin{array}{rcl}
    j(\omega_{SU(1,2)},\rho_{can.}) & = & -e^1\otimes e^1 -e^2\otimes e^2
    -e^3\otimes e^3 -e^4\otimes e^4\\
    && + e^5\otimes e^5 +e^6\otimes e^6\:.
    \end{array}
    \end{equation}
\end{enumerate}

\bigskip

\begin{Conv}
\begin{enumerate}
    \item We call $J_{\rho}$ the \emph{complex structure that is
    associated to $\rho$} or shortly the \emph{associated complex
    structure}.
    \item We call $j(\omega,\rho)$ the \emph{metric that is
    associated to $(\omega,\rho)$} or shortly the \emph{associated
    metric}. We denote it by $g_6$, since we will also work with
    metrics on 7- or 8-dimensional spaces.
\end{enumerate}
\end{Conv}

We remark that the objects that we have defined are related by the
formula

\begin{equation}
\label{omegaequation} \omega(v,w) := g_6(v,J_{\rho}(w))\:.
\end{equation}

We can decide if a pair $(\omega,\rho)$ determines an $SU(3)$- or
$SU(1,2)$-structure without referring to a special basis.

\begin{Th}
\label{SU3StrThm} Let $V$ be a 6-dimensional real vector space and
let $\omega\in \bigwedge^2 V^{\ast}$ and $\rho\in \bigwedge^3
V^{\ast}$ be stable. Moreover, let $J_{\rho}$ and $g_6$ be defined
as above. We assume that $\omega$ and $\rho$ satisfy the equations

\begin{enumerate}
    \item $\omega\wedge\rho = 0$,
    \item $J^{\ast}_{\rho}\rho \wedge \rho = \tfrac{2}{3}
    \omega\wedge\omega\wedge\omega$.
\end{enumerate}

If in this situation

\begin{enumerate}
    \item $g_6$ has signature $(6,0)$ and $J_{\rho}$ is a complex
    structure, $(\omega,\rho)$ is an $SU(3)$-structure.
    \item $g_6$ has signature $(2,4)$ and $J_{\rho}$ is a complex
    structure, $(\omega,\rho)$ is an $SU(1,2)$-structure.
\end{enumerate}
\end{Th}

\begin{Rem}
\begin{enumerate}
    \item Since $J_{\rho}^{\ast}\rho \wedge \rho$ and
    $\tfrac{2}{3} \omega\wedge\omega\wedge\omega$ are both
    6-forms, the second condition from the theorem is a
    normalization of the pair $(\omega,\rho)$.
    \item If $(\omega,\rho)$ is a pair of stable forms satisfying
    $\omega\wedge\rho = 0$ and $J^{\ast}_{\rho}\rho \wedge \rho =
    \tfrac{2}{3} \omega\wedge\omega\wedge\omega$ and it is not
    an $SU(3)$- or $SU(1,2)$-structure, $J_{\rho}$ is a
    para-complex structure and $(\omega,\rho)$ is an
    $SL(3,\mathbb{R})$-structure.
\end{enumerate}
\end{Rem}

The reason for the above considerations is to define 
$G$-structures on manifolds.

\begin{Def}
\label{SU3DefII} Let $M$ be a 6-dimensional manifold, $\omega\in
\bigwedge^2 T^{\ast}M$, and $\rho\in \bigwedge^3 T^{\ast}M$.
Moreover, let $G\in \{SU(3),SU(1,2)\}$. $(\omega,\rho)$ is called
a \emph{$G$-structure on $M$} if for all $p\in M$
$(\omega_p,\rho_p)$ is a $G$-structure on $T_pM$.
\end{Def}

\begin{Conv}
Since the endomorphism field $J_{\rho}$ in general has torsion, we
call it the \emph{almost} complex structure on $M$.
\end{Conv}

\subsection{$G$ is $G_2$ or $G_2^{\ast}$}
\label{G2Subsection}

With help of the concepts from the previous subsection we
are able to define $G_2$- and $G_2^{\ast}$-structures.

\begin{DaL}
\label{G2Def} We supplement the basis $(e_i)_{i=1,\ldots,6}$ of
$\mathbb{R}^6$ with $e_7$ to a basis of $\mathbb{R}^7$. The form

\begin{enumerate}
    \item $\phi_{G_2} := \omega_{SU(3)}\wedge e^7 + \rho_{can.}$ is
    stabilized by $G_2$.
    \item $\phi_{G_2^{\ast}} := \omega_{SU(1,2)}\wedge e^7 + \rho_{can.}$ is
    stabilized by $G_2^{\ast}$.
\end{enumerate}

$G_2$ denotes the compact real form of the complex Lie group
$G_2^{\mathbb{C}}$ and $G_2^{\ast}$ denotes the split real form.
Let $V$ be a 7-dimensional real vector space and $\phi$ be a
3-form on $V$. If there exists a basis $(v_i)_{i=1,\ldots,7}$ of
$V$ such that with respect to $(v_i)_{i=1,\ldots,7}$

\begin{enumerate}
    \item $\phi$ can be identified with $\phi_{G_2}$,
    $\phi$ is called a \emph{$G_2$-structure}.
    \item\label{G2ast1} $\phi$ can be identified with $\phi_{G_2^{\ast}}$,
    $\phi$ is called a \emph{$G_2^{\ast}$-structure}.
\end{enumerate}
\end{DaL}

\begin{Rem}
\label{G2DefRemark} There are exactly two open orbits of the
action of $GL(7,\mathbb{R})$ on $\bigwedge^3 \mathbb{R}^{7\ast}$
\cite{Reichel,Schouten}. Their union is a dense subset of
$\bigwedge^3 \mathbb{R}^{7\ast}$. One orbit consists of all
3-forms that are stabilized by $G_2$ and the other one consists
of all 3-forms that are stabilized by $G_2^{\ast}$.
\end{Rem}

Any $G_2$- or $G_2^{\ast}$-structure on a vector space $V$
determines a symmetric non-degenerate bilinear form $g_7$ and a
volume form $\text{vol}_7$. As in the previous subsection,
there are explicit rational $GL(7,\mathbb{R})$-equivariant maps
$\bigwedge^3_s V^{\ast} \rightarrow S^2(V^{\ast})$ and
$\bigwedge^3_s V^{\ast} \rightarrow \bigwedge^7 V^{\ast}$ that
assign $g_7$ and $\text{vol}_7$ to $\phi$. The explicit definition
of these maps can be found in \cite{Cortes}. The tensors $\phi$,
$g_7$, and $\text{vol}_7$ are related by the formula

\begin{equation}
\label{AssMetricG2} g_7(v,w)\:\text{vol}_7 = \tfrac{1}{6}
(v\lrcorner \phi) \wedge (w\lrcorner \phi) \wedge \phi
\quad \forall v,w\in V\:.
\end{equation}

Analogously to Subsection \ref{SU3Subsection}, we have

\begin{Le}
\label{g7Lemma}
Let $V$ be a 7-dimensional real vector space and $\phi$ be a
stable 3-form on $V$.

\begin{enumerate}
    \item If $\phi$ is a $G_2$-structure, $g_7$ has signature $(7,0)$.
    In particular, $g_7$ is the Euclidean metric on $\mathbb{R}^7$ if 
    $\phi$ coincides with $\phi_{G_2}$.

    \item If $\phi$ is a $G_2^{\ast}$-structure, $g_7$ has signature 
    $(3,4)$. In particular, $g_7=g_6 + e^7\otimes e^7$ if $\phi$ 
    coincides with $\phi_{G_2^{\ast}}$.
\end{enumerate}
\end{Le}

We can relate $\text{vol}_7$ to the 3-forms on the 6-dimensional
subspace $\text{span}(v_i)_{i=1,\ldots,6}$.

\begin{Le}
Let $\phi$ be a $G_2$- or $G_2^{\ast}$-structure on a vector space
$V$ and $(v_i)_{i=1,\ldots,7}$ be a basis of $V$ with the properties
from Definition and Lemma \ref{G2Def}. On 
$\text{span}(v_i)_{i=1,\ldots,6}$ there exists a canonical $SU(3)$-
or $SU(1,2)$-structure $(\omega,\rho)$ and we have

\begin{equation}
\text{vol}_7 = \tfrac{1}{4} J^{\ast}_{\rho}\rho \wedge 
\rho \wedge v^7\:.
\end{equation}

In particular, $\text{vol}_7$ is $e^{1234567}$ if $\phi$ is
$\phi_{G_2}$ or $\phi_{G_2^{\ast}}$.
\end{Le}

$g_7$ and $\text{vol}_7$ determine a Hodge-star operator 
$\ast$ on $\bigwedge^{\ast} V^{\ast}$.

\begin{Le}
Let $\phi$ be a $G_2$- or $G_2^{\ast}$-structure. 
The 4-form $\ast\phi$ is stable and can be described as

\begin{equation}
v^7 \wedge J_{\rho}^{\ast}\rho + \tfrac{1}{2}\omega \wedge
\omega\:.
\end{equation}
\end{Le}

\begin{Conv}
We call $g_7$ ($\text{vol}_7$, $\ast\phi$) the \emph{metric 
(volume form, 4-form) that is associated to $\phi$}. 
\end{Conv}

We define $G_2$- and $G_2^{\ast}$-structures on manifolds as
in the previous subsection.

\begin{Def}
\label{G2DefII} Let $M$ be a 7-dimensional manifold and $\phi\in
\bigwedge^3 T^{\ast}M$. Moreover, let $G\in \{G_2, G_2^{\ast}\}$.
$\phi$ is called a \emph{$G$-structure on $M$} if for all $p\in M$
$\phi_p$ is a $G$-structure on $T_pM$.
\end{Def}

\subsection{$G$ is Spin($7$) or $\text{Spin}_0(3,4)$} In this final
subsection, we introduce Spin($7$)- and $\text{Spin}_0(3,4)$-structures.

\begin{DaL}
\label{Spin7Def} We supplement the basis $(e_i)_{i=1,\ldots,7}$ of
$\mathbb{R}^7$ with $e_8$ to a basis of $\mathbb{R}^8$. The form

\begin{enumerate}
    \item $\Phi_{\text{Spin}(7)}:= e^8 \wedge \phi_{G_2} + \ast\phi_{G_2}$
    is stabilized by Spin($7$).
    \item $\Phi_{\text{Spin}_0(3,4)}:= e^8 \wedge \phi_{G_2^{\ast}} +
    \ast\phi_{G_2^{\ast}}$ is stabilized by the identity component
    $\text{Spin}_0(3,4)$ of $\text{Spin}(3,4)$.
\end{enumerate}

Let $V$ be an 8-dimensional real vector space and $\Phi$ be a
4-form on $V$. If there exists a basis $(v_i)_{i=1,\ldots,8}$ of
$V$ such that with respect to $(v_i)_{i=1,\ldots,8}$

\begin{enumerate}
    \item $\Phi$ can be identified with $\Phi_{\text{Spin}(7)}$,
    $\Phi$ is called a \emph{Spin($7$)-structure}.
    \item\label{Spin341} $\Phi$ can be identified with $\Phi_{\text{Spin}_0(3,4)}$,
    $\Phi$ is called a \emph{$\text{Spin}_0(3,4)$-structure}.
\end{enumerate}
\end{DaL}

Analogously to Subsection \ref{SU3Subsection} and \ref{G2Subsection},
any Spin($7$)- or $\text{Spin}_0(3,4)$-structure determines a
symmetric non-degenerate bilinear form $g_8$ and a volume form
$\text{vol}_8$. $\text{vol}_8$ is given by $\tfrac{1}{14}\Phi
\wedge \Phi$ and $g_8$ satisfies a slightly more complicated
relation as (\ref{AssMetricG2}), which can be found in Karigiannis
\cite{Kar}.

Unlike $\omega$, $\rho$, and $\phi$, $\Phi$ is not a stable form.
Nevertheless, we have similar results as in the previous two
subsections.

\begin{Le}
\label{g8Lemma}
Let $\Phi$ be a Spin($7$)- or $\text{Spin}_0(3,4)$-structure. In the
first case $g_8$ has signature $(8,0)$ and in the second case it has
signature $(4,4)$. In particular, $g_8$ is the Euclidean metric on 
$\mathbb{R}^8$ if $\Phi$ coincides with $\Phi_{\text{Spin}(7)}$ and
$g_8=g_7 + e^8\otimes e^8$ if $\Phi$ coincides with 
$\Phi_{\text{Spin}_0(3,4)}$. In both cases, we have

\begin{equation}
\text{vol}_8 = \text{vol}_7 \wedge v^8\:.
\end{equation}
\end{Le}

\begin{Conv}
As in the previous subsections, we call $g_8$ the \emph{associated
metric} and $\text{vol}_8$ the \emph{associated volume form}.
\end{Conv}

\begin{Rem}
\begin{enumerate}
    \item $\Phi$ is self-dual with respect to $g_8$ and 
    $\text{vol}_8$.

    \item Any 4-form on an 8-dimensional real vector space that
    is stabilized by Spin($7$) or $\text{Spin}_0(3,4)$ is
    a Spin($7$)- or $\text{Spin}_0(3,4)$-structure.
    However, there is no simple criterion like Theorem
    \ref{SU3StrThm} that decides if a given 4-form is a
    Spin($7$)- or $\text{Spin}_0(3,4)$-structure.
\end{enumerate}
\end{Rem}

The notion of a Spin($7$)- or a $\text{Spin}_0(3,4)$-structure on
an 8-dimensional mani\-fold can be defined completely analogously
to Definition \ref{SU3DefII} and \ref{G2DefII}.

\section{Hitchin's flow equation}

One motivation to study $G$-structures is their relation
to metrics with special holonomy.

\begin{Def}
\label{TorFree}
Let $G\in\{\text{Spin}(7),\text{Spin}_0(3,4)\}$ and let $\Phi$ be a
$G$-structure on an 8-dimensional manifold. $\Phi$ is called 
\emph{torsion-free} if $d\Phi=0$.
\end{Def}

\begin{Le} 
Let $G$ be as above. The holonomy group of the metric that is 
associated to a torsion-free $G$-structure is a subgroup of $G$.
Conversely, let $(M,g)$ be a semi-Riemannian manifold
whose holonomy is contained in $G$. Then there exists a
torsion-free $G$-structure on $M$ whose associated metric is $g$.
\end{Le}

\begin{proof} See \cite{Fer} for $G=\text{Spin}(7)$ and \cite{Br} 
for $G=\text{Spin}_0(3,4)$.
\end{proof}

\begin{Rem}
There are analogous results for $G\in\{SU(3),SU(1,2),G_2,G_2^{\ast}\}$.
\end{Rem}

We also need the following $G$-structures with torsion.

\begin{Def}
\begin{enumerate}
    \item Let $(\omega,\rho)$ be an $SU(3)$- or $SU(1,2)$-structure
    on a 6-dimensional manifold. $(\omega,\rho)$ is called 
    \emph{half-flat} if $d\rho=0$ and $d\omega\wedge\omega = 0$. 

    \item Let $\phi$ be a $G_2$- or $G_2^{\ast}$-structure on a 
    7-dimensional manifold. $\phi$ is called \emph{cocalibrated} 
    if $d\ast\phi=0$.
\end{enumerate}
\end{Def}

Compact Riemannian manifolds with holonomy Spin($7$)
are hard to construct. However, many non-compact examples with
cohomogeneity one are known \cite{Baz,Baz2,BrS,Cve,KanI,KanII,Rei}.
All of the these metrics can be obtained by a method that was 
developed by Hitchin \cite{Hitchin}. As in the previous section, our 
presentation of the issue is similar as in \cite{Cortes}.

\begin{Th} \label{HitchinThm} (See \cite{Cortes,Hitchin})
Let $N^7$ be a 7-dimensional manifold and $U\subset N^7\times
\mathbb{R}$ be an open neighborhood of $N^7\times \{0\}$. Furthermore,
let $G\in\{G_2,G_2^{\ast}\}$ and $\phi$ be a cocalibrated $G$-structure
on $N^7$. Finally, let $\phi_t$ be  a one-parameter family of 3-forms
such that $\phi_t$ is defined on $U\cap (N^7\times\{t\})$. We assume 
that $\phi_t$ is a solution of the initial value problem

\begin{equation}
\label{HitchinEvol}
\begin{array}{rcl}
\tfrac{\partial}{\partial t} \ast_7 \phi_t & = & d_7\phi_t \\
\phi_0 & = & \phi \\
\end{array}
\end{equation}

The index "$7$" emphasizes that we consider $\ast$ and $d$ as
operators on $U\cap (N^7\times \{t\})$ instead of $U$.  If $U$ is 
sufficiently small, $\phi_t$ is a $G$-structure for all $t$ with 
$U\cap (N^7\times \{t\}) \neq \emptyset$. Moreover, it is 
cocalibrated for all $t$. The 4-form

\begin{equation}
\Phi := dt\wedge \phi_t + \ast_7\phi_t
\end{equation}

is a torsion-free Spin($7$)-structure if $G=G_2$ and a
torsion-free $\text{Spin}_0(3,4)$-structure if $G=G_2^{\ast}$. Let
$g_8$ be the metric that is associated to $\Phi$ and $g_t$ be the
metric on $N^7\times \{t\}$ that is associated to $\phi_t$. With
this notation we have

\begin{equation}
g_8 = g_t + dt^2\:.
\end{equation}
\end{Th}

\begin{Rem}
\label{HitchinRemark}
\begin{enumerate}
    \item The equation $\tfrac{\partial}{\partial t} \ast_7 \phi_t =
    d_7\phi_t$ is called \emph{Hitchin's flow equation}.
    Since $\ast_7$ depends non-linearly on $\phi_t$, it is a 
    non-linear partial differential equation.

    \item If $N^7$ and $\phi_0$ are real analytic, the system
    (\ref{HitchinEvol}) has a unique maximal solution that is
    defined on an open neighborhood of $N^7\times
    \{0\}$ \cite{Cortes}. This is a consequence of the
    Cauchy-Kovalevskaya Theorem. We assume from now 
    that all initial data are analytic. If the initial data are
    smooth but non-analytic, examples can be found where no
    short-term solution of (\ref{HitchinEvol}) exists \cite{Br2010}.  

    \item If $N^7$ is in addition compact, there exists a unique 
    maximal open interval $I$ with $0\in I$ such that the solution
    is defined on $N^7\times I$. 

    \item Let $f:N^7\rightarrow N^7$ be a diffeomorphism, $I$
    an interval with $0\in I$, $U = N^7\times I$, and $\phi_t$ be
    a solution of Hitchin's flow equation on $U$. In this situation,
    the pull-back $f^{\ast}\phi_t$ is also a solution with the initial 
    value $f^{\ast}\phi_0$. 
\end{enumerate}
\end{Rem}

There are analogous results for the relationship 
between half-flat $SU(3)$- or $SU(1,2)$-structures and parallel
$G_2$- or $G_2^{\ast}$-structures. The evolution equations

\begin{equation}
\begin{array}{rcl}
\tfrac{\partial}{\partial t} \rho_t & = & d\omega_t \\
\left(\tfrac{\partial}{\partial t} \omega_t\right) \wedge \omega_t
& = & dJ_{\rho_t}^{\ast} \rho_t \\
\end{array}
\end{equation} 

yield a one-parameter family of half-flat $SU(3)$- or 
$SU(1,2)$-struc\-tures on a 6-dimensional manifold 
$N^6$ if the initial value is half-flat. The
3-form $\omega_t\wedge dt + \rho_t$ is a parallel $G_2$-
or $G_2^{\ast}$-structure on an open neighborhood of
$N^6\times \{0\}$ in $N^6\times \mathbb{R}$. A proof
of these facts for the $SU(3)$-case can be found in \cite{Hitchin}
and a for the $SU(1,2)$-case in \cite{Cortes}.

Moreover, it is known that an $SU(2)$-structure on a
5-dimensional ma\-nifold that satisfies certain conditions can
always be embedded into a not necessarily complete Calabi-Yau
threefold \cite{ContiS}. 

The results that we have introduced in this section suggest the
following more general questions. Let $M^n$ be an $n$-dimensional
manifold with some kind of special geometry. What is the  
geometric structure that is induced on hypersurfaces $N^{n-1}$ of $M^n$?
Conversely, can any $(n-1)$-dimensional manifold that is equipped with 
that kind of geometric structure be embedded into a suitable $M^n$? These 
questions are studied in \cite{Br2010},\cite{Conti},\cite{Stock1}, and 
\cite{Stock2}. Since we restrict ourselves to the dimension $n=8$, we 
will not go into further details, but refer the reader to the cited literature.

\section{Proof of the main theorem}

In this section, we consider a 6-dimensional manifold $N^6$ that
carries an $SU(3)$- or $SU(1,2)$-structure $(\omega_0,\rho_0)$.
Our aim is to construct a parallel Spin($7$)- or 
$\text{Spin}_0(3,4)$-structure $\Phi$ on a tubular neighborhood
of the zero section of the trivial bundle $N^6\times \mathbb{R}^2$
such that the restriction of $\Phi$ to $N^6$ is $(\omega_0,\rho_0)$ 
in a suitable sense. More precisely, let $\epsilon>0$ be sufficiently 
small and  

\begin{equation}
B_{\epsilon}(0) := \{(x,y)\in\mathbb{R}^2 | x^2 + y^2 < \epsilon^2\}\:.
\end{equation}

We denote $N^6\times \{0\}\subset N^6\times B_{\epsilon}(0)$ 
shortly by $N^6$. On that submanifold we want to have 

\begin{equation}
\Phi = \tfrac{1}{2}\omega_0\wedge\omega _0+ dx\wedge\rho_0 + 
dy\wedge J_{\rho_0}^{\ast}\rho_0 + dx\wedge dy\wedge\omega_0
\end{equation}

or equivalently

\begin{equation}
\label{SU3N6}
\begin{array}{rcl}
\tfrac{\partial}{\partial y} \lrcorner \left(
\tfrac{\partial}{\partial x} \lrcorner \Phi\right) & = & \omega_0 \\
\tfrac{\partial}{\partial x}\lrcorner \Phi - 
dy \wedge \omega_0 & = & \rho_0 
\end{array}
\end{equation}

Our first step is to construct a $G_2$- or $G_2^{\ast}$-structure 
$\phi$ on

\begin{equation*}
V_{\epsilon} := N^6 \times \{(0,y)\in\mathbb{R}^2 | y^2< \epsilon^2 \}
\end{equation*}

that satisfies 

\begin{equation}
\phi = \rho + dy\wedge \omega \quad
\text{and} \quad d\ast\phi =0
\end{equation}

for a $y$-dependent $SU(3)$- or $SU(1,2)$-structure 
$(\omega,\rho)$ on $N^6$. Next, we insert $\phi$ as 
initial condition into Hitchin's flow equation, where $x$ 
plays the role of the coordinate $t$ in Theorem 
\ref{HitchinThm}. After that, we have finally found our 
$\Phi$. We describe how to construct the 3-form 
on $V_{\epsilon}$. The Hodge dual of $\phi$ is

\begin{equation}
\ast\phi = \frac{1}{2}\omega\wedge\omega + dy\wedge 
J_{\rho}^{\ast} \rho\:.
\end{equation}

$\phi$ is thus cocalibrated if and only if 

\begin{equation}
\label{7dimflow}
\begin{array}{rcl}
\left(\frac{\partial}{\partial y}\omega\right) \wedge \omega & = &
d J_{\rho}^{\ast} \rho \\
d\omega \wedge \omega & = & 0 \\
\end{array}
\end{equation}

for all $y$. In the above equation, $d$ denotes the exterior derivative
on the 6-dimensional manifold $N^6\times \{(0,y)\}$. We see that 
any choice of $\rho$ satisfies the system (\ref{7dimflow}). Since 

\begin{equation}
(\omega\wedge\omega)_y = \omega_0\wedge\omega_0 +
2 \int_0^y d J_{\rho}^{\ast} \rho\: d\widetilde{y} 
\end{equation}

and $d^2=0$, $d\omega\wedge\omega = 0$ is satisfied for
all $y$ if it is satisfied for $y=0$. Of course, $(\omega,\rho)$
shall be an $SU(3)$- or $SU(1,2)$-structure for all
$y\in (-\epsilon,\epsilon)$. Therefore, the system 
that $(\omega,\rho)$ has to satisfy is in fact

\begin{equation}
\begin{array}{rcl}
\left(\frac{\partial}{\partial y}\omega\right) \wedge \omega & = &
d J_{\rho}^{\ast} \rho \\
\omega\wedge\rho & = & 0 \\
2\:\omega^3 & = & 3\: J_{\rho}^{\ast} \rho\wedge\rho \\
\end{array}
\end{equation}

If we take the derivative of the last two equations with respect to 
$y$, we obtain the following system of first order differential 
equations 

\begin{equation}
\label{G2System}
\begin{array}{rcl}
\left(\frac{\partial}{\partial y}\omega\right) \wedge \omega & = &
d J_{\rho}^{\ast} \rho \\
\left(\frac{\partial}{\partial y}\rho\right) \wedge \omega  
+ \rho\wedge \left(\frac{\partial}{\partial y}\omega \right) & = & 0 \\
3\left(\frac{\partial}{\partial y} J_{\rho}^{\ast}\rho\right)\wedge \rho
+ 3J_{\rho}^{\ast}\rho\wedge \left(\frac{\partial}{\partial y} \rho \right)
- 6 \left(\frac{\partial}{\partial y} \omega \right)\wedge \omega^2 
& = & 0 \\ 
\end{array}
\end{equation}

with the initial conditions

\begin{equation}
\label{G2Initial}
\begin{array}{rcl}
d\omega_0 \wedge \omega_0 & = & 0 \\
\omega_0 \wedge \rho_0 & = & 0 \\
2\omega_0^3 & = & 3 J_{\rho_0}^{\ast} \rho_0\wedge\rho_0 \\
\end{array}
\end{equation}

Since all forms in a neighborhood of $\omega_0$ or $\rho_0$
are stable, any solution of (\ref{G2System}) and (\ref{G2Initial})
describes a $G_2$- or $G_2^{\ast}$-structure if $\epsilon$ is
sufficiently small. Let $z^1,\ldots,z^6$ be coordinates on an open
subset $U\subset N^6$. The system (\ref{G2System}) consists of 
$22$ equations for the $35$ coefficient functions of $\omega$ 
and $\rho$. It can be written as

\begin{equation}
\label{ImplicitEquations}
F\left(\omega,\rho,
\frac{\partial\omega}{\partial z^1},\ldots,
\frac{\partial\omega}{\partial z^6},
\frac{\partial\rho}{\partial z^1},\ldots,
\frac{\partial\rho}{\partial z^6},
\frac{\partial \omega}{\partial y},
\frac{\partial \rho}{\partial y}\right)=0\:.
\end{equation}

$\omega$ is up to the sign uniquely determined by 
$\omega^2$ \cite{Cortes,Hitchin}. The first equation of
(\ref{G2System}) thus fixes the value of $\tfrac{\partial 
\omega}{\partial y}$. The second and third equation restrict 
$\rho$ at each $p\in U$ to the set $S$ of all 
$\rho$ that satisfy $\omega\wedge\rho=0$ and $2\omega^3 =
3 J_{\rho}^{\ast} \rho\wedge\rho$. 

We prove that $S$ is a smooth manifold and determine its
dimension. The equation $\omega
\wedge\rho = 0$ is a linear condition on $\rho$. It follows 
from Schur's lemma that the image of the map $\alpha
\mapsto \omega\wedge\alpha$ is either trivial or all
of $\bigwedge^5 T_p^{\ast} U$. The first case can easily 
be excluded and the space of all $\rho$ that satisfy the
above condition thus has dimension $14$. Let $\varphi: 
\bigwedge^3_s T_p^{\ast} U \rightarrow \bigwedge^7
T_p^{\ast} U$ be defined by $\varphi(\rho) = J_{\rho}^{\ast} \rho
\wedge\rho$. In \cite{Cortes} it is proven that

\begin{equation}
(d\varphi)_{\rho}(\alpha) = 2  J_{\rho}^{\ast} \rho \wedge \alpha\:.
\end{equation}

$(d\varphi)_{\rho}$ has rank $0$ or $1$. Since $(d\varphi)_{\rho}(\rho) = 
2  J_{\rho}^{\ast} \rho \wedge \rho$, its rank is $1$ and $S$ is a 
manifold of dimension $13$. $(dF)_{(\frac{\partial
\omega}{\partial y}, \frac{\partial \rho}{\partial y})}$
therefore has maximal rank. The metric that is associated 
to $(\omega,\rho)$ induces a metric on $\bigwedge^3
T_p^{\ast} U$. We denote the orthogonal projection of a 
stable 3-form to the tangent space of $S$ by $\pi_{\omega}$. 
Our next step is to add the equation

\begin{equation}
\label{AdditionalEquation}
\pi_{\omega}\left(\frac{\partial \rho}{\partial y} \right) = 0
\end{equation}

to (\ref{G2System}). We obtain a system of type 
(\ref{ImplicitEquations}), where $F$ is replaced by
a an $\widetilde{F}$ that satisfies

\begin{equation}
\text{rk} (d\widetilde{F})_{(\frac{\partial \omega}{\partial t},
\frac{\partial \rho}{\partial t})} = 35\:.
\end{equation}

With help of the implicit function theorem, the extended
system can be rewritten to 

\begin{equation}
\begin{array}{rcl}
\frac{\partial\omega}{\partial y} & = & F_1\left(
\omega,\rho,
\frac{\partial\omega}{\partial x^1},\ldots,
\frac{\partial\omega}{\partial x^6},
\frac{\partial\rho}{\partial x^1},\ldots,
\frac{\partial\rho}{\partial x^6} \right) \\
\frac{\partial\rho}{\partial y} & = & F_2\left(
\omega,\rho,
\frac{\partial\omega}{\partial x^1},\ldots,
\frac{\partial\omega}{\partial x^6},
\frac{\partial\rho}{\partial x^1},\ldots,
\frac{\partial\rho}{\partial x^6}
\right) \\
\end{array}
\end{equation}

Since $N^6$ is a real analytic manifold, $F_1$ and $F_2$ are 
analytic, too. As in \cite{Cortes}, the Cauchy-Kovalevskaya
theorem guarantees that the extended system has a unique 
solution on an open neighborhood of $N^6 \subset N^6\times
\mathbb{R}$. Thus, (\ref{G2System}) has at least one solution 
on the same open set. If $N^6$ is compact, the solution exists 
on $V_{\epsilon}$ for a sufficiently small $\epsilon>0$. With 
help of Theorem \ref{HitchinThm}, we are finally able to prove
our main theorem.

\begin{Th} \label{MainThm}
Let $N^6$ be an analytic compact 6-manifold and let 
$(\omega_0,\rho_0)$ be an analytic $SU(3)$- or $SU(1,2)$-structure 
with $d\omega_0 \wedge \omega_0= 0$ on $N^6$. Then, there 
exists an $\epsilon>0$ and a parallel Spin($7$)- or 
$\text{Spin}_0(3,4)$-structure $\Phi$ on $N^6\times 
B_{\epsilon}(0)$ such that on $N^6\times \{0\}$ we have 

\begin{equation}
\begin{array}{rcl}
\tfrac{\partial}{\partial y}\lrcorner \tfrac{\partial}{\partial x}
\lrcorner \Phi & = & \omega_0 \\
\tfrac{\partial}{\partial x}\lrcorner \Phi - dy \wedge \omega_0 & = & \rho_0 \\
\end{array}
\end{equation}

where $x$ and $y$ are the standard coordinates on $B_{\epsilon}(0)$.
\end{Th}

\section{An example}

In this section, we show that the 4-form $\Phi$ from Theorem
\ref{MainThm} is not uniquely determined by the initial value 
$(\omega_0,\rho_0)$. Before we start, we define what we mean
by uniqueness in this situation. 

\begin{Def}
\label{EquivDef}
Let $\Phi_1$ and $\Phi_2$ be two Spin($7$)- or 
$\text{Spin}_0(3,4)$-structures on $N^6\times 
B_{\epsilon}(0)$ such that on $N^6\times \{0\}$ we have

\begin{equation}
\begin{array}{rclcl}
\tfrac{\partial}{\partial y}\lrcorner \tfrac{\partial}{\partial x}
\lrcorner \Phi_1 & = & \tfrac{\partial}{\partial y}\lrcorner 
\tfrac{\partial}{\partial x} \lrcorner \Phi_2 & =: & \omega_0\\
\tfrac{\partial}{\partial x}\lrcorner \Phi_1 - dy \wedge \omega_0 & = &
\tfrac{\partial}{\partial x}\lrcorner \Phi_2 - dy \wedge \omega_0  
&& \\
\end{array}
\end{equation}

We call $\Phi_1$ and $\Phi_2$ \emph{equivalent} if there exists
a diffeomorphism $f$ of $N^6\times B_{\epsilon}(0)$ that is the 
identity on $N^6\times \{0\}$ and satisfies $f^{\ast}\Phi_1 = 
\Phi_2$. Analogously, let $\phi_1$ and $\phi_2$ be  
$G_2$- or $G_2^{\ast}$-structures on $N^6 \times
(-\epsilon,\epsilon)$ such that on $N^6\times \{0\}$ we have

\begin{equation}
\label{G2InitialEx}
\begin{array}{rclcl}
\frac{\partial}{\partial y} \lrcorner \phi_1 & = & 
\frac{\partial}{\partial y} \lrcorner \phi_2 & =: & 
\omega_0 \\
\phi_1 - dy\wedge\omega_0 & = & 
\phi_2 - dy\wedge\omega_0 && \\
\end{array}
\end{equation}

$\phi_1$ and $\phi_2$ are called \emph{equivalent} if there exists
a diffeomorphism of $N^6\times (-\epsilon,\epsilon)$ with
the same properties as above.
\end{Def}

We restrict ourselves to the Riemannian case. For our example, 
$(\omega_0,\rho_0)$ shall be torsion-free. In other words, $N^6$ 
together with the initial $SU(3)$-structure is a Calabi-Yau manifold. 
Our strategy is to construct a one-parameter family of $G_2$-structures 
$\phi_{\delta}$ on $N^6\times S^1$ such that the standard coordinate 
$\theta\in \mathbb{R}/2\pi\mathbb{Z}$ of $S^1$ plays the role of 
$y$. After that, we consider Hitchin's flow equation with initial value 
$\phi_{\delta}$ in order to obtain 4-forms $\Phi_{\delta}$. Let $\alpha$ 
be a closed 3-form on $N^6$. We define a $G_2$-structure 
$\phi_{\delta}$ on $N^6\times S^1$ by

\begin{equation}
\label{G2ExampleFormula}
\phi_{\delta} = \omega_0\wedge d\theta  - J_{\rho_0}^{\ast}\rho_0
+ \delta\cdot\sin{\theta}\cdot \ast_6\alpha\:,
\end{equation}

where $\ast_6$ is the Hodge-star on $N^6$. We have 

\begin{equation}
\ast\phi_{\delta} = d\theta \wedge (\rho_0 
+ \delta\cdot \sin{\theta}\cdot \alpha)
+ \tfrac{1}{2} \omega_0\wedge\omega_0\:.
\end{equation}

Since $\phi_0$ is a $G_2$-structure, $\phi_{\delta}$ is also
a $G_2$-structure if $\delta$ is sufficiently small. Moreover,
$\phi_{\delta}$ is cocalibrated  and at $\theta=0$ each term
of (\ref{G2InitialEx}) is independent of $\delta$. 
Let $g_6$ be the metric on $N^6$ that is
associated to $(\omega_0,\rho_0)$ and $g_{8,\delta}$ 
be the metric on $N^6\times S^1\times (-\epsilon,
\epsilon)$ that is associated to $\Phi_{\delta}$. 
Since $\phi_0$ and $\Phi_0$ are both
torsion-free, we have $g_{8,0} = g_6
+ d\theta^2 + dx^2$ and the second fundamental form $II$ 
of $N^6\times \{(0,0)\}$ vanishes. If we find 
an $\alpha$ such that $II\neq 0$, $\Phi_0$
and $\Phi_{\delta}$ are non-equivalent. 

Let $X$ be a unit vector field on $N^6$. $X$ can be lifted to a
vector field on the product $N^6\times S^1\times (-\epsilon,\epsilon)$.
Outside of $N^6\times \{(0,0)\}$, $X$ is in general not a
unit vector field anymore. For all $\alpha$, $\tfrac{
\partial}{\partial \theta}$ is a unit normal field of 
$N^6\times \{(0,0)\}$. Since $[X,\tfrac{\partial}{\partial \theta}]=0$, 
we have on $N^6\times \{(0,0)\}$

\begin{equation}
\begin{aligned}
g\left(II(X,X),\tfrac{\partial}{\partial \theta} \right)
& = g\left(\nabla_X X,\tfrac{
\partial}{\partial \theta} \right) \\
& = \tfrac{1}{2}\left(Xg(X,\tfrac{
\partial}{\partial \theta}) +
Xg(\tfrac{
\partial}{\partial \theta},X) 
-\tfrac{
\partial}{\partial \theta} g(X,X) \right) \\
& = -\tfrac{1}{2} \tfrac{
\partial}{\partial \theta} g(X,X)\:.
\end{aligned}
\end{equation}

Since we can prescribe the value of a closed 3-form
at a fixed point arbitrarily, there exists an $\alpha$
such that the last term of the above equation does not 
vanish globally if $\delta>0$. We thus have proven that 
$\Phi_0$ and $\Phi_{\delta}$ are non-equivalent, although they 
share the same initial values.

\section{Outlook}

Let $N^6$ be a 6-dimensional manifold and $M^8$ be
an arbitrary $\mathbb{R}^2$-bundle over $N^6$. For reasons of brevity, we 
denote the zero section of $M^8$ also by $N^6$. We check under which
conditions $M^8$ admits a not necessarily parallel Spin($7$)- or 
$\text{Spin}_0(3,4)$-structure $\Phi$. 

First, we assume that a Spin($7$)-structure $\Phi$ exists on $M^8$. 
Let $\pi:M^8\rightarrow N^6$ be the projection map and $\pi^{-1}(U)$
with $U\subset N^6$ be the image of a local trivialization. Moreover, 
let $e_x$ and $e_y$ be orthonormal vertical vector fields on $\pi^{-1}(U)$
and $(e^x,e^y)$ be the duals of $(e_x,e_y)$ with respect to the metric. If 
we replace in equation (\ref{SU3N6}) $(\tfrac{\partial}{\partial x},
\tfrac{\partial}{\partial y})$ by $(e_x,e_y)$ and $dy$ by $e^y$, we obtain 
an $SU(3)$-structure $(\omega,\rho)$ on $U$. However, the $SU(3)$-structure 
can in general not be extended to all of $N^6$, since the bundle may not 
admit two global linearly independent sections. 

Spin($7$) acts transitively on the set of all oriented 6-dimensional 
subspaces of $\mathbb{R}^8$. The subgroup that fixes a subspace is 
isomorphic to $U(3)$. Therefore, any 6-dimensional oriented submanifold
of a Spin($7$)-manifold carries a canonical $U(3)$-structure and this is 
the most natural kind of geometry to suppose on $N^6$. In terms of
tensor fields, a $U(3)$-structure is defined by a non-degenerate 2-form
$\omega$, a Riemannian metric $g$ and an almost complex structure
$J$ such that $\omega(X,Y)= g(X,J(Y))$ for all vector fields $X$ and $Y$.  
In our situation, the $U(3)$-structure is determined by $\omega := 
e_y\lrcorner e_x\lrcorner \Phi$ and the restriction of the associated metric 
to the tangent space of $N^6$. Our definition of $\omega$ is independent 
of the choice of $(e_x,e_y)$ and $\omega$ is thus globally defined. The 
$\text{Spin}_0(3,4)$-case is completely analogous, since 
$\text{Spin}_0(3,4)/U(1,2)$ is the Grassmannian of all positive 
oriented planes in $\mathbb{R}^{4,4}$.

We return to the local situation. The restriction of the 4-form to the
subset $U$ of the zero section can be written as

\begin{equation}
\Phi = \tfrac{1}{2}\omega\wedge\omega + e^x\wedge\rho + 
e^y\wedge J_{\rho}^{\ast}\rho + e^x\wedge e^y\wedge\omega\:.
\end{equation}

We choose another $\pi^{-1}(\widetilde{U})$ and vertical vector fields 
$\widetilde{e}_x$ and $\widetilde{e}_y$ on $\widetilde{U}$ with
the same properties as above. Moreover, we assume that $U\cap 
\widetilde{U} \neq \emptyset$. On $\widetilde{U}$ we have

\begin{equation}
\Phi = \tfrac{1}{2}\widetilde{\omega}\wedge\widetilde{\omega} + 
\widetilde{e}^x\wedge\widetilde{\rho} + \widetilde{e}^y\wedge 
J_{\widetilde{\rho}}^{\ast} \widetilde{\rho} + \widetilde{e}^x\wedge \widetilde{e}^y\wedge\widetilde{\omega}
\end{equation}

for another $SU(3)$- or $SU(1,2)$-structure $(\widetilde{\omega},
\widetilde{\rho})$. On the intersection $\pi^{-1}(U\cap\widetilde{U})$
we have

\begin{equation}
\label{NewFrame}
\begin{array}{rcr}
\widetilde{e}_x & = & \cos{\theta}\cdot e_x + \sin{\theta}\cdot e_y \\
\widetilde{e}_y & = & -\sin{\theta}\cdot e_x + \cos{\theta}\cdot e_y \\
\end{array}
\end{equation}

for a function $\theta:U\cap\widetilde{U}\rightarrow\mathbb{R}$. Both terms for
$\Phi$ coincide only if

\begin{equation}
\begin{array}{rcr}
\widetilde{\rho} & = & \cos{\theta}\cdot\rho + \sin{\theta}\cdot J_{\rho}^{\ast}\rho \\
J_{\widetilde{\rho}}^{\ast}\widetilde{\rho} & = & 
-\sin{\theta}\cdot\rho + \cos{\theta}\cdot J_{\rho}^{\ast}\rho \\
\end{array}
\end{equation}

The transition functions for the bundle $M^8$ thus have to
be transition functions for the bundle $\bigwedge^{3,0}T^{\ast}N^6$,
too. In other words, $M^8$ has to be isomorphic to the canonical 
bundle of $N^6$ with respect to the almost complex structure
$J$. 

Conversely, we assume that there exists a line bundle isomorphism
$\eta: M^8 \rightarrow \bigwedge^{3,0} T^{\ast}N^6$ and that 
$N^6$ carries a $U(3)$- or $U(1,2)$-structure $(\omega,g,J)$. We 
choose local trivializations $\varphi_{\alpha}: U_{\alpha} \times 
\mathbb{R}^2 \rightarrow \pi^{-1}(U_{\alpha}) \subseteq M_8$ such 
that the transition functions have values in $SO(2)$. Let $x$ and $y$ be
the standard coordinates of $\mathbb{R}^2$. There exist unique 
one-forms $e^1$ and $e^2$ such that $\varphi_{\alpha}^{\ast}(e^1) 
= dx$ and $\varphi_{\alpha}^{\ast}(e^2) = dy$. If the $U_{\alpha}$ 
are sufficiently small, there exists a $(3,0)$-form $\rho$ on 
$U_{\alpha}$ such that $(\omega,\rho)$ is an $SU(3)$- or 
$SU(1,2)$-structure whose associated metric and almost complex 
structure coincide with $g$ and $J$. Any other $(3,0)$-form with the 
same properties as $\rho$ can be written as

\begin{equation}
\cos{\sigma_{\alpha}}\cdot\rho + \sin{\sigma_{\alpha}}\cdot J^{\ast}_{\rho} \rho
\end{equation} 

for a function $\sigma_{\alpha}:U_{\alpha} \rightarrow \mathbb{R}$.
We choose $\sigma_{\alpha}$ such that 

\begin{equation}
\begin{array}{lr}
\eta^{\ast -1}(e^1)(\cos{\sigma_{\alpha}}\cdot\rho + 
\sin{\sigma_{\alpha}}\cdot J^{\ast}_{\rho} \rho) & >0 \\
\eta^{\ast -1}(e^1)(-\sin{\sigma_{\alpha}}\cdot\rho + 
\cos{\sigma_{\alpha}}\cdot J^{\ast}_{\rho} \rho) & =0 \\
\end{array}
\end{equation}

and define a 4-form

\begin{equation}
\Phi = \tfrac{1}{2} \pi^{\ast}\omega \wedge \pi^{\ast}\omega + e^1\wedge 
\pi^{\ast}\rho + e^2\wedge \pi^{\ast}J^{\ast}_{\rho}\rho + e^1\wedge e^2\wedge 
\pi^{\ast}\omega
\end{equation}

on $\pi^{-1}(U_{\alpha})$. $\Phi$ is a Spin($7$)- or 
$\text{Spin}_0(3,4)$-structure. By a similar argument as before, we 
can prove that $\Phi$ is globally defined. The above observations 
yield the following lemma.

\begin{Le}
Let $M^8$ be an $\mathbb{R}^2$-bundle over a manifold $N^6$ that 
admits a $U(3)$- or $U(1,2)$-structure $(\omega,g,J)$. $M^8$ admits a 
Spin($7$)- or $\text{Spin}_0(3,4)$-structure if and only if $M^8$ is
isomorphic to the canonical bundle of $N^6$.    
\end{Le}  

We therefore propose the following conjecture.

\begin{Conj}
Let $N^6$ be an analytic compact 6-dimensional manifold with 
an analytic $U(3)$- or $U(1,2)$-structure $(\omega,g,J)$ that satisfies 
$d\omega \wedge \omega = 0$. Then there exists a parallel Spin($7$)- 
or $\text{Spin}_0(3,4)$-structure $\Phi$ on a tubular 
neighborhood of the zero section of the canonical bundle of 
$N^6$ such that

\begin{enumerate}
    \item the restriction of the associated metric to $N^6$ 
    coincides with $g$ and

    \item $e_y \lrcorner (e_x\lrcorner \Phi) = \omega$ for
    any two orthonormal vertical vector fields $e_x$ and $e_y$ 
    along $N^6$. 
\end{enumerate}
\end{Conj}

Theorem \ref{MainThm} yields a parallel Spin($7$)- or
$\text{Spin}_0(3,4)$-structures $\Phi_{\alpha}$ on each
set of type $\varphi_{\alpha}(U_{\alpha} \times 
B_{\epsilon_{\alpha}}(0))$ for a sufficiently small 
$\epsilon_{\alpha} > 0$. Since we have added 
equation (\ref{AdditionalEquation}) to our system, which 
makes its solution unique, the $\Phi_{\alpha}$ are in a
certain sense canonical. It would be nice if we could
glue them together to a global Spin($7$)- or
$\text{Spin}_0(3,4)$-structure and thus prove our
conjecture. 

This idea works only if the $\Phi_{\alpha}$ are compatible
with the transition functions $\tau_{\alpha\beta}: U_{\alpha} 
\cap U_{\beta} \rightarrow U(1)$. More precisely, let $x$ 
and $y$ be vertical coordinates on $\pi^{-1}(U_{\alpha})$ 
such that $x$ is mapped to $y$ by $i\in U(1)$. Moreover, we 
introduce coordinates $\widetilde{x}$ and $\widetilde{y}$ on 
$\pi^{-1}(U_{\beta})$ with the same properties. 
On $\pi^{-1}(U_{\alpha} \cap U_{\beta})$
both coordinates are related by an equation that is analogous 
to (\ref{NewFrame}). $\Phi_{\alpha}$ and $\Phi_{\beta}$  
should coincide on $\pi^{-1}(U_{\alpha}\cap U_{\beta})$. In 
particular, this should be the case if $\tau_{\alpha\beta}$ is
constant. In this situation, the restriction of $\Phi_{\alpha}$ 
to $\pi^{-1}(U_{\alpha}\cap U_{\beta})$ is obtained as the 
solution of Hitchin's flow equation with a $G_2$-structure 
$\phi_{\alpha}$ on

\begin{equation}
V_{\alpha}:=(U_{\alpha}\cap U_{\beta}) \times \{(0,y)\in\mathbb{R}^2 | 
y^2 < \min{\{\epsilon_{\alpha},\epsilon_{\beta}\}}^2 \} 
\end{equation}

as initial value. Analogously, the restriction of $\Phi_{\beta}$
to $\pi^{-1}(U_{\alpha}\cap U_{\beta})$ is obtained as the solution
of Hitchin's flow equation with a $G_2$-structure $\phi_{\beta}$ on

\begin{equation}
V_{\beta}:=(U_{\alpha}\cap U_{\beta}) \times \{(\sin{\tau}\cdot y,
\cos{\tau}\cdot y)\in\mathbb{R}^2 | y^2 < 
\min{\{\epsilon_{\alpha},\epsilon_{\beta}\}}^2 \} 
\end{equation}

as initial value, where $\tau$ is the constant value of
$\tau_{\alpha\beta}$. Let $f_{\tau}$ be the diffeomorphism of 
$\pi^{-1}(U_{\alpha}\cap U_{\beta})$ that is defined by 

\begin{equation}
f_{\tau}(p,x,y) := (p,\cos{\tau}\cdot x + \sin{\tau}\cdot y, 
-\sin{\tau}\cdot x + \cos{\tau}\cdot y)\:.
\end{equation}

We restrict $f_{\tau}$ to a map $V_{\alpha}\rightarrow V_{\beta}$.
Since it does not make a difference if we choose the set on which 
we construct the $G_2$-structure as $V_{\alpha}$ or $V_{\beta}$, 
we have $\phi_{\alpha} = f_{\tau}^{\ast}\phi_{\beta}$. Therefore, 
we also have $\Phi_{\alpha} = f_{\tau}^{\ast}\Phi_{\beta}$ for any 
value of$\tau$. The Spin($7$)- or $\text{Spin}_0(3,4)$-structure
$\Phi$ that we obtain by glueing thus has to be preserved by 
$f_{\tau}$. The differential of $f_{\tau}$ at a point of $U_{\alpha}
\cap U_{\beta}$ can be identified with the complex matrix 
$A_{\tau}:=\text{diag}(1,1,1,e^{i\tau})$. Unfortunately, conjugation 
by $A_{\tau}$ does not preserve Spin($7$) or $\text{Spin}_0(3,4)$ 
if we interpret it as a real $8\times 8$-matrix. Therefore, we 
cannot have $\Phi = f_{\tau}^{\ast}\Phi$ and our conjecture cannot 
be proven by this simple idea.

For the same reason we cannot make $\Phi$ unique by assuming 
that the standard $U(1)$-action on the canonical bundle leaves 
$\Phi$ invariant.  Therefore, the $U(3)$- or $U(1,2)$-structure on 
$N^6$ cannot be extended to a $U(1)$-invariant parallel 
Spin($7$)- or $\text{Spin}_0(3,4)$-structure. This is a striking 
difference to \cite{Biel}, where the fact that 
$\text{diag}(1,\ldots,1,e^{i\tau})$ commutes with $SU(n)$ allows 
the existence of a $U(1)$-invariant $SU(n)$-structure on the 
canonical bundle.


\begin{thebibliography}{999}
    \bibitem{Baz} Bazaikin, Ya.V.: On the new examples of complete noncompact Spin(7)-holonomy
    metrics. Sib. Math. J. 48, No.1, 8-25 (2007).

    \bibitem{Baz2} Bazaikin, Ya V.; Malkovich, E.G.: Spin(7)-structures on complex linear bundles 
    and explicit Riemannian metrics with holonomy group SU(4). Sb. Math. 202, No. 4, 467-493 
    (2011). 

    \bibitem{Biel} Bielawski, R.: Ricci-flat K\"ahler metrics on canonical bundles. Math. Proc.
    Cambridge Phil. Soc. 132, 471 - 479 (2002).

    \bibitem{Br} Bryant, R.: Metrics with exceptional holonomy. Ann. of Math. 126, 525-576
    (1987).

    \bibitem{BrS} Bryant, R.; Salamon, S.: On the construction of some complete metrics
    with exceptional holonomy. Duke Mathematical Journal 58, 829-850 (1989).

    \bibitem{Br2010} Bryant, R.: Non-embedding and non-extension results in special
    holonomy. In: The many facets of geometry, Oxford Univ. Press, Oxford, 346-367 (2010). 

    \bibitem{ContiS} Conti, D.; Salamon, S.: Generalized Killing spinors in dimension 5. 
    Trans. Amer. Math. Soc. 359, 5319-5343 (2007).  

    \bibitem{Conti} Conti, D.: Embedding into manifolds with torsion. Math. Z. 268,  
    725-751 (2011). 

    \bibitem{Cortes} Cort\'{e}s, V.; Leistner, T.; Sch\"afer, L.; Schulte-Hengesbach, F.: Half-flat
    Structures and Special Holonomy. Proc. Lond. Math. Soc. (3) 102, No. 1, 113-158 (2011).

    \bibitem{Cve} Cveti\v{c}, M.; Gibbons, G.W.; L\"u, H.; Pope, C.N.: Cohomogeneity one manifolds
    of Spin(7) and $G_{2}$ holonomy. Ann. Phys. 300 No.2, 139-184 (2002). 

    \bibitem{Fer} Fern\'{a}ndez, M.: A classification of Riemannian manifolds with structure
    group Spin($7$). Ann. Mat. Pura Appl., IV. Ser. 143, 101-122 (1986).

    \bibitem{Hitchin} Hitchin, N.: Stable forms and special metrics. In: Fern\'andez, Marisa
    (editor) et al.: Global differential geometry: The mathematical legacy of Alfred Gray. Proceedings
    of the international congress on differential geometry held in memory of Professor Alfred Gray.
    Bilbao, Spain, September 18-23 2000. / AMS Contemporary Mathematical series 288, 70-89
    (2001). 

    \bibitem{KanI} Kanno, H.; Yasui, Y.: On Spin(7) holonomy metric based on 
    $\text{SU}(3)/\text{U}(1)$ I. J. Geom. Phys. 43 No.4, 293-309 (2002). 

    \bibitem{KanII} Kanno, H.; Yasui, Y.: On Spin(7) holonomy metric based on
    $\text{SU}(3)/\text{U}(1)$ II. J. Geom. Phys. 43 No.4, 310-326 (2002). 

    \bibitem{Kar} Karigiannis, S.: Deformations of $G_2$ and Spin($7$)-structures. Can. J. Math.
    57 No. 5, 1012-1055 (2005). 

    \bibitem{Reichel} Reichel, W.: \"Uber die Trilinearen Alternierenden Formen in 6 und 7
    Ver\"anderlichen. Dissertation, Greifswald 1907.

    \bibitem{Rei} Reidegeld, F.: Exceptional holonomy and Einstein metrics constructed from
    Aloff-Wallach spaces. Proc. Lond. Math. Soc. (3) 102, No. 6, 1127-1160 (2011).

    \bibitem{Schouten} Schouten, J.A.: Klassifizierung der alternierenden Gr\"o\ss en dritten Grades
    in 7 Dimensionen. Rend. Circ. Mat. Palermo 55, 137-156 (1931).

    \bibitem{Stock1} Stock, S.: Gauge Deformations and Embedding Theorems for 
    Special Geometries. Preprint, arXiv:0909.5549v2 [math.DG]. 

    \bibitem{Stock2} Stock, S.: Evolution of Geometries with Torsion. Dissertation,
    Mathema\-tisches Institut der Universit\"at zu K\"oln, 2011.
\end{thebibliography}
\end{document}